\documentclass{amsart}

\usepackage[latin1]{inputenc}
\usepackage{amsfonts}
\usepackage{amsmath}
\usepackage{amsthm}
\usepackage{amssymb}
\usepackage{latexsym}
\usepackage{enumerate}

\newtheorem{theorem}{Theorem}[section]
\newtheorem{lemma}[theorem]{Lemma}
\newtheorem{proposition}[theorem]{Proposition}
\newtheorem{corollary}[theorem]{Corollary}
\newtheorem{fact}[theorem]{Fact}
\newtheorem{remark}[theorem]{Remark}

\newtheorem{definition}[theorem]{Definition}

\numberwithin{equation}{section}

\begin{document}

\newcommand{\cc}{\mathfrak{c}}
\newcommand{\N}{\mathbb{N}}
\newcommand{\PP}{\mathbb{P}}
\newcommand{\forces}{\Vdash}
\newcommand{\R}{\mathbb{R}}
\newcommand{\LL}{\mathbb{L}}
\newcommand{\MM}{\mathbb{M}}

\title{A $C(K)$ BANACH SPACE WHICH DOES NOT HAVE THE SCHROEDER-BERNSTEIN PROPERTY}
\author{PIOTR KOSZMIDER}
\thanks{The  author was partially supported by Polish Ministry of
Science and Higher Education research grant N N201
386234 and Junta de Andaluc\'\i a and FEDER grant P06-FQM-01438. The research leading to this
paper was conducted while the author was visiting Universidad de Granada in the spring of
2009. He would like
to thank Professors Miguel Mart\'\i n and Javier Mer{\'\i} for their hospitality and fruitful
discussions.}
\address{Institute of Mathematics, Polish Academy of
Sciences, ul. \'Sniadeckich 8,
P.O. Box 21,
00-956,  
 Warsaw, Poland}
\address{Instytut Matematyki Politechniki \L\'odzkiej,
ul.\ W\'olcza\'nska 215, 90-924 \L\'od\'z, Poland}
\address{Departamento de An\'{a}lisis Matem\'{a}tico \\ Facultad de
 Ciencias \\ Universidad de Granada \\ 18071 Granada, Spain}

\email{\texttt{piotr.math@gmail.com}}

%
%
%
%
\begin{abstract} We construct a totally disconnected compact Hausdorff space
$K_0$ which has  clopen subsets $K_2\subseteq K_1\subseteq K_0$ such that
 $K_0$ is homeomorphic to $K_2$ and hence $C(K_0)$ is isometric as a Banach space
to $C(K_2)$ but
$C(K_0)$ is not isomorphic to $C(K_1)$. This gives two nonisomorphic Banach spaces of the form
$C(K)$ which are isomorphic to complemented subspaces of each other (even in the above strong
isometric sense),
providing a solution to the Schroeder-Bernstein problem for Banach spaces
of the form $C(K)$. $K_0$ is obtained as
a particular compactification of the  pairwise disjoint union of a sequence of $K$s for which
$C(K)$s have few operators. 
\end{abstract}

\maketitle

\markright{}
\vskip 26pt
\section{Introduction}
\vskip 13pt
If $X, Y$ are Banach spaces, then $X\sim Y$ will mean that they
are isomorphic, and $X\overset{c}{\hookrightarrow} Y$ that $Y$ has a complemented subspace isomorphic to
$X$.
 A. Pe\l czy\'nski has proved (see \cite{projections}, \cite{casazza})
that if two Banach spaces $X$, $Y$ satisfy
$X\sim X^2$, $Y\sim Y^2$, then  the following version of the Schroeder-Bernstein theorem holds for them:

$$\hbox{\rm If}\ \ X\overset{c}{\hookrightarrow} 
Y\ \ \hbox{\rm and} \ \ Y\overset{c}{\hookrightarrow} X,\ \ \hbox{\rm then} \ \ X\sim Y.$$

The problem whether this holds in general, without demanding that $X$ and $Y$ are 
isomorphic to their squares has been known as the Schroeder-Bernstein problem 
for Banach spaces (\cite{casazza}). 
After several decades it has been solved in the negative by T. Gowers in \cite{gowers}.
In this paper we give quite different construction of a Banach space which solves this problem
in the negative which additionally is a classical Banach space of real valued continuous
functions on a compact Hausdorff space with the supremum norm.

One of the main ingredients of Gowers' solution was the use of Banach spaces $X_{GM}$, 
obtained by Gowers himself and Maurey (\cite{gowersmaurey}), with
few operators in the sense that each bounded, linear operator on it is of the form
$\lambda Id+S$ where $S$ is strictly singular and $\lambda$ is a scalar.
The space of \cite{gowers} is an ``exotic" completion of $c_{00}(X_{GM}\oplus X_{GM})$ where
the norm in the sum is defined depending on the choice of a  basic sequence
in $X_{GM}$. Also the space $X_{GM}$ has quite a complex definition of the norm.
Other examples of Banach spaces which solve the Schroeder-Bernstein problem
in the negative were also based on spaces like $X_{GM}$  and were constructed 
 by Gowers and Maurey \cite{gowersmaurey2},  Galego
(\cite{eloi1}, \cite{eloi2}, \cite{eloi3}) or Galego and Ferenczi \cite{eloi4}.

When we deal with Banach spaces of  continuous functions on a compact Hausdorff space $K$,
we face, however, the simplest possible norm, the supremum norm and nontrivial examples are obtained
through nontrivial compact $K$s. In \cite{few} we considered Banach spaces 
of continuous functions with few operators in the sense that any operator on them  is
of the form $g Id+S$ where $g\in C(K)$ and $S$ is weakly compact (equivalently here, a
strictly singular operator) or such that the adjoint of any operator is of the form
$g Id+ S$ where $g$ is a Borel bounded function and $S$ is weakly compact on the dual.
The latter construction and its connected
modification was used in \cite{few} to obtain results analogous to
 some of the results of Gowers and Maurey but in the class of $C(K)$ Banach spaces, like
indecomposable Banach spaces or spaces nonisomorphic to their hyperplanes. The former
construction was obtained in \cite{few} under the assumption of the continuum hypothesis
but later was modified in \cite{plebanek} or \cite{iryna} to avoid this assumption.
Of course the former form of the operator is simpler and we will make use
of $C(K)$ Banach spaces with few operators in this sense.
For our purpose the construction of \cite{iryna} is the most convenient
because its $K$ is perfect, separable and totally disconnected and
the Boolean algebra  $Clop(K)$ of all clopen subsets of $K$ has a dense subalgebra
with many automorphisms.
On the other hand  by the results of \cite{centripetal} the
approach where we deal with spaces with few adjoint operators in the above sense
would have to work as well
but would be formally more complicated.

The general plan of the construction and what should work follows \cite{gowers}
with $X_{GM}$ replaced by a $C(K)$ with few operators.
However, all this must be achieved in the  $C(K)$-environment,
that is through an appropriate compactification of the topological disjoint sum of
compact $K_i$s rather than the definition of a norm. Note, also, that
if $K$ is infinite and metrizable then $C(K)$ is isomorphic to its square, so
we must deal with nonseparable $C(K)$s.
So, the realization of this general plan is quite far from \cite{gowers}.
Another complication is
that our spaces have few operators in a different sense that those of \cite{gowersmaurey},
this seems to be an obstacle which caused that we were unable to obtain  spaces which 
are isomorphic to their cube but not isomorphic to their square as in the paper \cite{gowers}.
If there could be such spaces of the form $C(K)$ still remains an open problem.

The main idea is to define a Banach space $X_+$ as a $C(K_+)$ where $K_+$ is an
appropriate compactification of the pairwise disjoint sum of clopen subsets  $K_n=K_{1,n}\cup K_{2,n}$ 
for $n\in \N$,
such that all operators from $C(K_n)$ into itself are of the form $fId+S$ where
$f\in C(K_n)$ and $S$ is weakly compact and all $K_n$ are pairwise homeomorphic.
It follows that all operators from $C(K_{i,n})$ into $C(K_{3-i,m})$ are
weakly compact, where $i=1,2$ and $n,m\in\N$.

$Y_+=\{f\in X_+: f|K_{1,0}=0\}\sim C(K_+\setminus K_{1,0})$. It is clear that
$Y_+$ is complemented in  $X_+$ and the compactification is defined
in such a way that $\{f\in X_+: f|K_1=0\}$ is a subspace of $Y_+$ isomorphic to $X_+$
and clearly complemented in $Y_+$. Thus we only face the proof of the fact that
$X_+$ and $Y_+$ are not isomorphic. 

Although $C(K_{i,n})$ and $C(K_{3-i,m})$ are quite incomparable (there are only weakly compact
and so strictly singular operators among them), we exploit the fact that
both of the algebras of all clopen sets of $K_{1,n}$ and $K_{2,n}$ have isomorphic dense
subalgebras. This allows us to compare $f_i\in K_{i,n}$ for $i=1,2$ using an automorphism of
these dense subalgebras. Finally the compactification $K_+$ is defined in such a
way that each $f\in C(K_+)$ can be identified with a sequence
of functions $f_n\in C(K_n)$ such that $f_n|K_{1,n}$ and $f_n|K_{2,n}$ get ``closer" to each other 
(in a sense defined using the dense subalgebras and an automorphism between them)
when $n$ tends to  infinity. This construction and the notation used
to  deal with it is described in Section 3.

The point of the proof that $X_+$ is not isomorphic to $Y_+$ is to prove that
any  isomorphism $T$ between these spaces would create infinitely many ``bumps" i.e., 
$T(f_n)|K_{1,m(n)}$s and $T(f_n)|K_{2,m(n)}$s would not get closer
to each other when $n$ and some $m(n)$ tend to  infinity.
This would give $T(f)$ outside of $X_+$ leading
to a contradiction with the existence of such an isomorphism. At least this is what happens when we
look at the shift of the part of $X_+$ corresponding to $\bigcup_{n\in \N} K_{1,n}$
which should not be a well-defined operator on $X_+$. It turns out that
all hypothetic isomorphisms from $Y_+$ to $X_+$ would share this behaviour
and so there are none of them.
This argument is the subject of the last section 
where a strengthening of a well-known fact that $\{x\in K: T^*(\delta_x)(\{x\})\not=0\}$
is at most countable, if $T$ is weakly compact is also proved (\ref{lemmaweaklycompacttrick}).

To be able to detect these bumps one approximate an operator $T:X_+\rightarrow X_+$
by a matrix of multiplications by continuous functions. This approximation is
developed in Section 5 which relies on Section 4 where we prove general properties
of operators on $X_+$. Some of the ideas 
of Section 4 which we use for dealing with compactifications of disjoint $K_i$s where
for each $i\in \N$ the space $C(K_i)$ has few operators were developed in
\cite{kmm1} and \cite{kmm2}.

The next Section 2 gathers some classical results about $C(K)$s and some of their
corollaries needed in the rest of the paper.

\vskip 26pt

\section{Some fundamental results on $C(K)$ spaces}

In this section we recall some fundamental results of Bessaga, Grothendieck, Pe\l czy\'nski
and Rosenthal concerning the Banach spaces $C(K)$. We will do quote them in the form
most convenient for the sake of this paper, also we make some simple corollaries.

\begin{theorem}\label{theoremdrewnowski}\cite{drewnowski}(Cor. 2)
Let $E$ be a Banach space and   $X$ be either $l_\infty$ or $c_0$ and let $T:X\rightarrow E$ be a linear
continuous operator. Then exactly one of the two possibilities holds:
\begin{enumerate}
\item $T(e_n)\rightarrow 0$
\item There is an infinite subset $M\subseteq \N$ such that $T|X(M)$ is
an isomorphism.
\end{enumerate}
\end{theorem}

\begin{theorem}\label{theoremrosenthal2}\cite{rosenthal}(Thm 1.3)
Let $B$ be a Banach space and $X$ an injective space.
Let $T:X\rightarrow B$ be an operator such that there exists a subspace
$A$ of $X$ isomorphic to $c_0(\Gamma)$, with $T|A$ an isomorphism.
Then there exists a subspace $Y$ of $X$ isomorphic to $l_\infty(\Gamma)$
with $T|Y$ an isomorphism.  
\end{theorem}

\begin{definition}\label{definitiongeneratesczero}
 Let $K$ be a Hausdorff compact space.
Suppose that $(f_n)$ is a sequence from $C(K)$. We say that
it generates a copy of $c_0$ if and only if it is a basic sequence
which is equivalent to the standard basis of $c_0$.
\end{definition}

\begin{definition}
Let $K$ be a Hausdorff compact space. We say that $f,g\in C(K)$ are
disjoint if and only  if $fg=0$.
\end{definition}

\begin{fact}\label{factczero} Let $K$ be a Hausdorff compact space and $\varepsilon>0$.
Suppose that $(f_n)$ is a bounded, pairwise disjoint sequence of functions 
such that $||f_n||>\varepsilon>0$ for each $n\in\N$. Then $(f_n)$ generates 
a copy of $c_0$.
\end{fact}

\begin{fact}\label{factczerodisjoint} Let $K$ be a Hausdorff compact space.
Suppose that $(f_n)_{n\in\N}$ is a sequence of functions 
such that $(f_n)_{n\in\N}$ generates 
a copy of $c_0$ and suppose that $a_n\subseteq K$ are clopen sets such that
$||f_n\chi_{a_n}||>\varepsilon$ for some $\varepsilon>0$.
Then there is an infinite $M\subseteq\N$ such that $(f_n\chi_{a_n})_{n\in M}$ generates 
a copy of $c_0$
\end{fact}

\begin{proof} Let $T:c_0\rightarrow C(K)$ be an isomorphism
such that $T(1_{n})=f_n$.
Note that for each $x\in K$ we have 
$$\sum_{n\in\N}\alpha_n f_n\chi_{a_n}(x)=\sum_{n\in\N}\alpha_n^*(x) f_n(x),$$
where $\alpha_n^*(x)=\alpha_n$ if $x\in a_n$ and $\alpha_n^*(x)=0$ otherwise, and so,
 such an $(\alpha_n^*(x))_{n\in N}$ belongs to $c_0$.
For every $\varepsilon>0$ and every $k_1<k_2$ from $\N$ there is an $x\in K$ such that
$$||\sum_{n=k_1}^{k_2}\alpha_n f_n\chi_{a_n}||\leq ||\sum_{n=k_1}^{k_2}\alpha_n^*(x) f_n||+\varepsilon\leq $$
$$\leq ||T|| ||\alpha_{k_1}^*(x),...,\alpha_{k_2}^*(x)||_{c_0}+\varepsilon
\leq ||T|| ||\alpha_{k_1},...,\alpha_{k_2}||_{c_0}+\varepsilon.$$
So, if $(\alpha_n)\in c_0$, we have that $||\alpha_{k_1},...,\alpha_{k_2}||_{c_0}$
goes to zero if $k_1$ and $k_2$ go to infinity, hence we may conclude that 
$\sum_{n\in\N}\alpha_n f_n\chi_{a_n}$ converges uniformly and so the limit is
continuous. A similar argument shows that $||\sum_{n\in\N}\alpha_n f_n\chi_{a_n}||\leq ||T|| 
||(\alpha_n)||_{c_0}$.
So we may conclude that
$$T'((\alpha_n))=\sum_{n\in\N}\alpha_n f_n\chi_{a_n}$$
is a bounded linear operator defined on $c_0$ and into $C(K)$.
Now we can use \ref{theoremdrewnowski} since $||T'(1_n)||=||f_n\chi_{a_n}||>\varepsilon$.

\begin{theorem}\label{theorempelczynski}\cite{singular}(Thm. 1)
 A linear bounded  operator on $T:C(K)\rightarrow C(K)$ is weakly compact if and only if 
it is strictly singular.
\end{theorem}

\begin{theorem}\label{theoremdiesteluhl}\cite{diesteluhl}
A linear bounded  operator on $T:C(K)\rightarrow C(K)$ is weakly compact if and only if 
for every bounded pairwise disjoint sequence $(f_n)\subseteq C(K)$
we have $T(f_n)\rightarrow 0$.

\end{theorem}

\begin{theorem}\label{theoremrosenthal1}\cite{rosenthal}
Suppose that $\varepsilon>0$ and  for each $n,k\in\N$ we have non-negative $m_{nk}\in \R$
and  $\sum_{n\in\N} m_{nk}<\varepsilon$ for each $k\in \N$.
Then for each $\delta>0$ there is an infinite $M\subseteq \N$ such that
$$\sum_{n\in M\setminus\{k\}}m_{nk}<\delta$$
for each $k\in M$.
\end{theorem}

\begin{corollary}\label{corollarysupremus}
Suppose $(f_n)$ is a  bounded pairwise disjoint sequence of 
elements of a $C(K)$. Suppose $(\zeta_n)_{n\in \N}$ 
is a bounded sequence in $C^*(K)$ and let $\varepsilon>0$. There is an infinite $M\subseteq \N$
such that whenever $M'\subseteq M$ 
and  the supremum $f_{M'}=\sup_{n\in M'}f_n$ exists in $C(K)$,
then $|\zeta_k(f_{M'})-\zeta_k(f_k)|<\varepsilon$ if $k\in M'$ and
$|\zeta_k(f_{M'})|<\varepsilon$ if $k\in M\setminus M'$
\end{corollary}

\begin{proof}
Let $\mu_k$ be the Radon measure on $K$ corresponding to $\zeta_k$.
Define $m_{nk}=|\mu_k|(f_n)$. Apply the Rosenthal lemma \ref{theoremrosenthal1}
for $\delta=\varepsilon$ obtaining an infinite $M_1\subseteq M$ such that
then $\sum_{n\in M_1\setminus\{k\}}|\zeta_k(f_{n})|<\varepsilon$ if $k\in M_1$.
Thus, it is enough to obtain an infinite $M\subseteq M_1$ such that 
whenever $M'\subseteq M$ 
and  the supremum $f_{M'}=\sup_{n\in M'}f_n$ exists in $C(K)$,
then for each $k\in \N$
$$\sum_{n\in M'}\zeta_k(f_{n})=\zeta_k(f_{M'}).$$
Note that 
$$\sum_{n\in M'}\zeta_k(f_{n})=\int \sum_{n\in M'} f_n d\mu_k,$$
and 
$\zeta_k(f_{M'})=\int  f_{M'} d\mu_k$ where $\sum_{n\in M'} f_n$ is taken pointwise and
is possibly not in $C(K)$. Let $\{\N_\xi:\xi<\omega_1\}$
be a family of infinite sets of $\N$ whose pairwise intersections are finite.
If none of them works as  $M$, there is
a $k\in\N$ and there are infinite $b_\xi\subseteq \N_\xi$ for uncountably many $\xi\in\omega_1$
such that
$$\int (f_{b_\xi}-\sum_{n\in b_\xi} f_n) d\mu_k\not=0$$
but this is impossible since the Borel functions which we integrate above are pairwise disjoint
as shown in the claim of 5.2. of \cite{few}.

\end{proof}

\begin{theorem}\label{theorembessagapelczynski}\cite{bessagapelczynski}
Let $X$ be a Banach space $(x_n)\subseteq X$ be a basic sequence and
$(x^*_n)\subseteq X^*$ a sequence biorthogonal to $(x_n)$.
If $(y_n)\subseteq X$ fulfills the condition
$$\sum_{n\in\N}||x_n-y_n||||x_n^*||<\delta<1,$$
Then $(y_n)$ is a basic sequence and $(x_n)$ and $(y_n)$ are equivalent.
\end{theorem}

\end{proof}

\vskip 26pt
\section{The construction and the notation}
\vskip 13pt
The Stone functor from the category of Boolean algebras and homomorphisms into
the category of compact Hausdorff spaces with continuous functions will be denoted $S$.
Thus $S(A)$ is the Stone space of the algebra $A$ and $S(h):S(B)\rightarrow S(A)$
is the continuous mapping induced by a homomorphism $h:A\rightarrow B$
of Boolean algebras and is given by
$S(h)(x)= h^{-1}(x)$.  If $a$ is an element of a Boolean algebra $A$,
the basic clopen set $\{u\in S(A): a\in u\}$ of the Stone space of
$A$ corresponding to $a$ will be denoted by $[a]$. The Boolean algebra
of clopen subsets of a compact space $K$ is denoted by $Clop(K)$.
$S(Clop(K))$ will be identufied with $K$ and $Clop(S(A)$ with $A$.
For more on  the Stone duality see Chapter 3 of \cite{koppelberg}.

\vskip 6pt
\begin{lemma}\label{lemmaexistence}
There is an infinite Boolean algebra $A$ whose  Stone space
$K$ is a sum $K_{1,*}\cup K_{2,*}$ for 
 $K_{1,*}$ and $K_{2,*}$ disjont and  clopen which satisfies the following:
\begin{enumerate}
\item $K$ is  separable (compact totally disconnected) perfect space,
\item every operator on $C(K)$ is of the form $T=fId+S$ where $f\in C(K)$
and $S$ is weakly compact, 
\item $C(K)$ contains no copy of $l_\infty$,
\item there is a dense subalgebra $B\subseteq A$ such that $K_{1,*}, K_{2,*}\in B$ and an automorphism
$j: B\rightarrow B$ such that $j(K_{3-i,*})=K_{i,*}$ for $i=1,2$ and $j^2=Id_{B}$,
\item suppose that $(a_n)$ is a sequence of clopen and pairwise disjoint subsets of $K$.
 There is an infinite $M\subseteq \N$ such that
there is in  $A$ the supremum $\sup_{n\in M}a_n=a$.
\end{enumerate}
\end{lemma}
\begin{proof} Let $K$ be the Stone space of the algebra $A$ obtained in Chapter 3 of \cite{iryna}.
By 3.6.1, 3.6.2., 3.6.3., and the fact that $K$ has property ($K'$) (3.3.1.) we obtain (1) without
the perfectness, (2) and (5).
(3) follows from 2.4. of \cite{few} because every space which contain $l_\infty$ must have
its hyperplanes isomorphic to itself.
So, we are left with showing that 
there is a dense subalgebra $B$ and its automorphism $j$ as above.
$B$ will be a free Boolean algebra with $2^\omega$ independent generators. 
The algebra $A$ includes the algebra $B$ of clopen subsets of $\{0,1\}^{2^\omega}$
and is included the algebra of regular open subsets of $\{0,1\}^{2^\omega}$
hence $B$ is dense in $A$ (see \cite{iryna}. p. 57). As $B$ is 
a free Boolean algebra, it is clear that there is an automorphism and $K_{1,*}$, $K_{2,*}$
as required and that $K$ is prefect.

\end{proof}

\begin{remark} In the literature there are several
constructions of  spaces $C(K)$ where all operators are of the form $T=f Id+S$ 
where $f\in C(K)$ and $S$ is weakly compact: those obtained using some special set theoretic assumptions
in Section 6 of \cite{few}, or in \cite{rogerio} and having some additional properties or
without any additional set-theoretic assumptions spaces of \cite{plebanek} or \cite{iryna}.
The spaces of Section 6 of \cite{few} or of  \cite{rogerio} have countable subalgebras $B$ satysfying the above
theorem. The space \cite{plebanek} 
is only presented in the connected version and it is unclear to us
if its totally disconnected version would have some natural
dense subalgebras with appropriate automorphisms.
The separability of $K$ from \cite{iryna} 
is not used in any argument in this paper but it shows that
a $C(K)$ solution to the Schroeder-Bernstein problem could be a 
subspace of $l_\infty$. We conjecture that also the space of Section 3 of \cite{few}
obtained without any special set-theoretic assumptions 
can be used to obtain the main result of this paper, the results of
\cite{centripetal} support this conjecture, however it seems that it would
complicate the details. Thus it seems that the space of \cite{iryna} is the most optimal
for our purpose.
\end{remark}

\vskip 13pt Let $A$, $B$, $K$, $K_{1,*}$, $K_{2,*}$, $j$ be as in \ref{lemmaexistence}, moreover let

\begin{itemize}
\item $C=\{b\in B: j(b)=b\}=\{b\cup j(b): b\in B,\ [b]\subseteq K_{1,*}\}=
 \{j(b)\cup b: b\in B,\ [b]\subseteq K_{2,*}\}$

\item $L$ be the Stone space of $C$, 

\item $\tau=S(\subseteq): K\rightarrow L$ is the canonical surjection, where
$\subseteq: C\rightarrow A$ is the inclusion.

\item  $Z\equiv C(L)$ considered as a subspace
of $C(K)$ i.e., $Z=\{f\circ\tau: f\in C(L)\}$.
\end{itemize}

\begin{remark}\label{remarkbowtie}
Note that $\tau$ identifies, among others, the pairs $x,y$ of points of $K_{1,*}$ and $K_{2,*}$
respectively such that $S(j)(x\cap B)=y\cap B$. So in particular, for each $x\in K_{i,*}$ 
there is $y\in K_{3-i,*}$ such that $\tau(x)=\tau(y)$.
\end{remark}

As $K$ is  separable we may w.l.o.g. assume that
$A$ is  a subalgebra of $\wp(\N)$. For $n\in\N$ let 
\begin{itemize}
\item $\N_{n}$s be pairwise disjoint copies of $\N$,
\item $A_{n}$s be the copies
of $A$ in $\wp(\N_{n})$, 
\item $K_n$s be the (pairwise disjoint) Stone spaces of $A_n$s,
\item $K_{1,n}$ and $K_{2,n}$ be the copies of $K_{1,*}$ and $K_{2,*}$ in $K_{n}$,
\item $K^{-n}=\bigcup_{i\leq n}K_i$,
\item $B_n$ be the copy of $B$ in $A_n\subseteq \wp(\N_n)$,
\item $C_n$ be the copy of $C$ in $A_n\subseteq \wp(\N_n)$,
\item $L_n$ be the Stone space of $C_n$,
\item $\tau_n=S(\subseteq_n): K_n\rightarrow L_n$ be the canonical surjection, where
$\subseteq_n: C_n\rightarrow A_n$ is the inclusion.
\item  $j_{n}$ be the copy of $j$ on $B_n$.
\item $i_{n,*}: K\rightarrow K_n$, $i_{*,n}: K_n\rightarrow K$,    
$i_{m,n}: K_n\rightarrow K_m$ be homeomorphisms preserving  the above mentioned  objects respectively,

\item $h_{n,*}: L\rightarrow L_n$, $h_{*,n}: L_n\rightarrow L$,    
$h_{m,n}: L_n\rightarrow L_m$ be homeomorphisms.

\end{itemize}

 We will consider the following Boolean algebras:
\begin{itemize}
\item $D_\infty=\{a\subseteq \bigcup_{n\in\N} \N_{n}: \forall n\in \N\ \  a\cap \N_{n}\in A_{n}\},$
\item $D_0=\{a\in D_\infty: \exists m\in \N \ ( \forall n>m
\ \ a\cap \N_{n}=\emptyset)\  \hbox{or}\  (\forall n>m
\ \ a\cap \N_{n}=\N_{n}) \},$
\item $D_+=\{a\in D_\infty: \exists m\in \N \forall n>m
\ \ a\cap \N_{n}\in C_{n}
          \}.$
\end{itemize}
In other words, $D_\infty$ is the product algebra of $A_n$s,
$D_0$ is the direct sum algebra of $A_n$s,
and finally $D_+$ is the algebra of  those elements of $D_\infty$ whose coordinates eventually belong to $C_n$.

Let 
\begin{itemize}
\item $K_\infty$, $K_0$,  $K_+$ be the Stone spaces of $D_\infty$, $D_0$,  and
$D_+$ respectively. 
\end{itemize}
And finally let 
\begin{itemize}
\item $X_\infty=C(K_\infty)$, $X_+=C(K_+)$, $Z_n\equiv C(L_n)$ considered as a subspace
of $C(K_n)$ i.e., $Z_n=\{f\circ\tau_n: f\in C(L_n)\}$.
\item $X_0=C_0(K_0)=\{f\in C(K_0): f(\infty)=0\}$, where $\infty$
is the only ultrafilter of $D_0$ which does not contain any $\N_n$.
\item $K^{+n}=K_+\setminus K^{-n}$.

\end{itemize}

We also will need the following notation for some natural
projections and inclusions, let
\begin{itemize}
\item  $P_{n}: C(K_+)\rightarrow C(K_{n})$ be the restriction to $K_{n}$. 
\item  $I_{n}: C(K_n)\rightarrow C(K_+)$ be the extension from $K_{n}$ by zero function.
\item  $P_{-n}: C(K_+)\rightarrow C(K^{-n})$ be the restriction to $K^{-n}$. 
\item  $I_{-n}: C(K^{-n})\rightarrow C(K_+)$ be the extension from $K^{-n}$ by zero function.
\item  $P_{+n}: C(K_+)\rightarrow C(K^{+n})$ be the restriction to $K^{+n}$.
\item  $I_{+n}: C(K^{+n})\rightarrow C(K_+)$ be the extension from $K^{+n}$ by zero function.
\end{itemize}

The Stone-Weierstrass theorem
which implies that the functions which assume only finitely many values on clopen sets from 
the algebra $A$ are dense in the $C(K)$ where $K$ is the Stone space of $A$, gives the
idea about which functions belong to the spaces $X_0, X_+, X_\infty$.
In particular we have the following three descriptions :
\begin{lemma}
$$X_\infty\equiv \{(f_n)_{n\in N}\in \Pi_{n\in \N} C(K_{n}): 
(||f_n||)_{n\in \N}\ \ \hbox{is bounded}\}.$$
The norm is the supremum norm. 
\end{lemma}

\begin{lemma}
$$X_0\equiv \{(f_n)_{n\in N}\in \Pi_{n\in \N} C(K_{n}): 
\lim_{n\rightarrow \N}||f_n||=0\}.$$
The norm is the supremum norm. 
\end{lemma}
For $f\in C(K_n)$ denote by $d_n(f)$ the distance
between $f$ and the subspace $Z_n\subseteq C(K_n)$. Thus we have:

\begin{lemma}\label{lemmacharacterizationxplus}
$$X_+\equiv \{(f_n)_{n\in N}\in \Pi_{n\in \N} C(K_{n}): (||f_n||)_{n\in \N}\ \ \hbox{is bounded and}\ 
\lim_{n\rightarrow\infty}
d_n(f_n)=0\}.$$
The norm is the supremum norm. 
\end{lemma}

We will often denote elements of the above spaces using the above representations i.e.,
as a sequence $(f_n)$. 
Then, under the appropriate identification
we have  $X_0\subseteq X_+\subseteq X_\infty$ which will be used as well.

\begin{definition} For $x,y\in K_n$ we say that $x\bowtie y$
if and only if   $\tau_n(x)=\tau_n(y)$.
\end{definition}

\begin{lemma}\label{lemmadistance} Suppose that $f\in C(K_n)$. Then 
$$d_n(f)={\sup\{ {{|f(x)-f(y)|}}:
x,y\in K_{n},\ 
   x\bowtie y   \}\over 2}.$$       
\end{lemma}
\begin{proof}
Let $s$ be the supremum from the lemma. If $g\in Z_n$, then
clearly $g(x)=g(y)$ for any $x,y$ such that $x\bowtie y$
 we obtain $d_n(f)\geq s/2$.

Now let $f\in C(K_n)$ and 
 $\delta>s$. 
It is possible to obtain a clopen finite partition $U_1,..., U_k$ of $L_n$
such that $diam(f[\tau_n^{-1}[U_m]])<\delta$ for each $m=1,..., k$. 
This follows from the fact that $diam(f[\tau_n^{-1}[\{t\}]])\leq s$ 
for each $t\in L_n$ and the continuity of $f$ as well as the compactness of the
spaces involved. Now define $g\in Z_n$ as 
a function which is constant on each $\tau_n^{-1}[U_m]$ and assumes on it 
the arithmetic average of the extrema of the values of $f[\phi^{-1}[U_m]])$.
This way $||f-g||\leq \delta/2$ and hence
$d_n(f)\leq s/2$.\par

\end{proof}

\begin{corollary}\label{corollary1assymptotic}
Suppose  that $g\in Z_n$ and that $f\in C(K_n)$    
is arbitrary such that
 $||f-g||< \varepsilon$. Suppose that $x,y\in K_{n}$  
are such that $x\bowtie y$. Then $|f(x)-f(y)|<2\varepsilon$.\par
\end{corollary}

\begin{proof}
Follows directly from the previous lemma.
\end{proof}

\begin{corollary}\label{corollary2assymptotic}

Suppose that $(f_n)_{n\in\N}\in X_+$ and $x_n,y_n\in K_{n}$,
are such that $x_n\bowtie y_n$. Then
$|f_n(x_n)-f_n(y_n)|$ converges to $0$ as $n$  tends to infinity.
\end{corollary}

\begin{proof}
Follows directly from the previous corollary and from Lemma \ref{lemmacharacterizationxplus}.
\end{proof}

Finally we will need the following notation:

\begin{itemize}

\item $Y_0=\{ (f_n)\in X_0: f|K_{1, 0}=0\}$
\item $Y_+=\{ (f_n)\in X_+: f|K_{1, 0}=0\}$.

\end{itemize}

\begin{lemma} $Y_+$ is complemented in $X_+$ and 
$X_+$ contains a complemented subspace isometric to $Y_+$. Both projections
of norm one. 
\end{lemma}
\begin{proof}
Clearly $X_+=C(K_+)\equiv C(K_{1,0})\oplus C(K_{2,0})\oplus \{ (f_n)\in X_+: f_1=0\}$
and $\{ (f_n)\in X_+: f_1=0\}$ is isometric to $X_+$
while $C(K_{2,0})\oplus \{ (f_n)\in X_0: f_1=0\}$ is isometric to $Y_+$.

\end{proof}

\noindent The rest of this paper is devoted to the proof that $Y_+$ 
is not isomorphic to $X_+$. 

\vskip 26pt
\section{Operators on $X_+$}
\vskip 13pt
First, we will define two operators which will serve to illustrate several
phenomena in Proposition \ref{lambdatheta} and Remarks \ref{remarkcompositions} and \ref{remarkmultiplications}.
\vskip 13pt
\begin{definition}\label{definitionlambdatheta}
Fix a pairwise disjoint sequence $(g_n)$ in $C(K)$ such
that $||g_n||=1$ and fix a dense countable subset $\{x_n: n\in\N\}$ of $K$ and for
each $n\in\N$ fix
$\theta_n\in C^*(K)$  such that
$||\theta_n||=1$ and such that $\theta_n|Z_n=0$.

\begin{itemize}
\item Define $\Lambda:C(K)\rightarrow X_+$ by $\Lambda(f)|K_n=f(x_n)\chi_{K_n}$
for $f\in C(K)$,
and 
\item
define $\Theta: X_+\rightarrow C(K)$ by $\Theta((f_n))=\sum_{n\in \N} \theta_n(f_n)g_n$.
\end{itemize}
\end{definition}

\begin{proposition}\label{lambdatheta} $\Theta$ and $\Lambda$ are well-defined linear bounded
operators. 
The image of $\Theta$ is isomorphic to $c_0$, in 
particular $X_+$ is not a Grothendieck space and so, it has a complemented copy of $c_0$. 

\end{proposition}
\begin{proof}
The only nontrivial part of the first statement is whether $\Theta((f_n))$ is always
in $C(K)$ for any $(f_n)\in X^+$. It is enough to note that $\theta_n(f_n)$s tend to
zero. But for each $n\in\N$ there is a $z_n\in Z_n$ such that $||f_n-z_n||$s go to zero.
Then using the fact that $\theta_n|Z_n=0$ we have

$$|\theta_n(f_n)|=|\theta_n(f_n-z_n)+\theta_n(z_n)|\leq ||\theta_n||||f_n-z_n||$$
which goes to zero. So, the last statement follows as well (see  \cite{schachermeyer} 5.1, 5.3).

\end{proof}

Note that $X=\{(f_n)\in X_+: \forall n\in\N\ f_n|K_n\ \hbox{\rm is constant}\}$ is a 
complemented in $X_+$ copy of $l_\infty$. Thus using it and the operators $\Theta$ and $\Lambda$
we may construct many nontrivial operators from $X^+$ into $X^+$ which have nothing to do with
any multiplications. 
It was already noted in \cite{kmm1} that nevertheless dealing with this
kind of spaces we have strong tools (see for example
\ref{corollarycompositions}) to analyze all the operators.

\begin{lemma}\label{lemmauniformity}
Suppose that $T: C(K)\rightarrow X_+$ is a linear bounded operator.
Let $(f_m)_{m\in \N}$ be a bounded sequence of  pairwise disjoint functions in $C(K)$.
Then
$$\forall\varepsilon \exists k \forall n> k\forall m  \ \ d_n(T({f_m})|K_n)\leq\varepsilon.$$
\end{lemma}
\begin{proof} We may w.l.o.g. assume that the norms of the sequence of functions are bounded by $1$.
Suppose that the lemma is false. Let $\varepsilon>0$ be such that for every $k\in \N$ there
is $n>k$ and $m_k\in \N$ such that 
$$d_n(T({f_{m_k}})|K_n)>\varepsilon.$$
Hence we can choose increasing $n_k$'s such that 
$$d_{n_k}(T({f_{m_k}})|K_{n_k})> \varepsilon.$$
Moreover as $T({f_{m_k}})\in X_+$ by \ref{lemmacharacterizationxplus}
we must have that $m_k$s
assume infinitely many values and so we may assume that 
$$d_{n_k}(T({f_{m_k}})|K_{n_k})>\varepsilon.$$
and $m_k$s are increasing. 

By \ref{lemmadistance} there are points $x_k,y_k\in K_{n_k}$
such that  $x_k\bowtie y_k$ and for each $k\in \N$ 
$$|T({f_{m_k}})(x_k)-T({f_{m_k)}}(y_k)|>2\varepsilon.$$
Let $$\mu_k=T^*(\delta_{x_k}-\delta_{y_k}).$$
So we have that $|\int f_{m_k}d\mu_k|>2\varepsilon$. Since $K$ is totally disconnected we can
choose pairwise disjoint clopen $[a_{m_k}]$ included in the supports of $f_{m_k}$ such that
$$|\mu_k([a_{m_k}])|>\varepsilon.$$

Now we  use \ref{corollarysupremus} for the pairwise disjoint bounded sequence
$(\chi_{[a_{m_k}]})_{k\in \N}$, $\varepsilon/2$ and the functionals equal to the measures $\mu_k$ 
and then \ref{lemmaexistence} (5) obtaining
an infinite $M\subseteq \N$ such that there is the supremum $a=\sup_{k\in M}a_{m_k}$ in $Clop(K)$
and $|\mu_k(a)-\mu_k(a_{m_k})|<\varepsilon/2$ for each $k\in M$. However this implies that
$| T(\chi_{[a]}) (x_k)-T(\chi_{[a]}) (y_k)|>\varepsilon/2$ for each $k\in M$ which by 
\ref{corollary2assymptotic} implies that $T(\chi_{[a]})$ is not in $X_+$, a contradiction.

\end{proof}

\begin{definition}\label{definitiondnsimple}
Let $f\in X_+$. We say that $f$ is $(C_n)$-simple if for every
$n\in \N$ the function $f|K_n$ assumes only finitely many values
on clopen sets  from $C_n$.
\end{definition}

\begin{lemma}\label{lemmadnsimple}
Suppose that $(f_m)$ is a sequence of elements of   $X_+$ are such that
\begin{enumerate}
\item $f_m$s are $(C_n)$-simple for each $m\in\N$
\item $f_m|K_{-m}=0$
\item $(f_m)$ generates a copy of $c_0$.
\end{enumerate}
Suppose that $T:X_+\rightarrow C(K)$ is a bounded linear operator.
Then $||T(f_m)||$s converge to $0$.
\end{lemma}
\begin{proof}
We will find an injective subspace of $X_+$  which
contains all the $f_m$'s and then assuming that the lemma is false we will
apply \ref{theoremdrewnowski} and  \ref{theoremrosenthal2} to obtain a contradiction with  \ref{lemmaexistence} (3).
Let $E_n$ be the  Boolean algebra of subsets of
$K_n$ generated by the preimages
under the functions $f_m$ where $m\in\N$. Note  that it is finite as almost all
the functions $f_m$ are zero on $K_n$ and the remaining functions are $(C_n)$-simple.
Consider the Boolean algebra
$$E=\{a\in D_\infty: \forall n\in \N \ \ a\cap\N_n\in E_n\},$$
and note that it is a complete Boolean algebra and hence the
Banach space $V$ of all continuous functions which are in the closure of finite linear combinations
of characteristic functions of elements of $E$ is included in $X_+$, is injective
and that $f_m\in V$ for each $m\in \N$. 

Now, if $||T(f_m)||$s do not converge to $0$, then by  \ref{theoremdrewnowski} 
we may w.l.o.g. assume that $T$ is an isomorphism on the copy of $c_0$
generated by $(f_m)$.
Finally apply \ref{theoremrosenthal2} to conclude that $C(K_n)$ contains a copy
of $l_\infty$ which gives a contradiction with  \ref{lemmaexistence} (3).

\end{proof}

\begin{lemma}\label{lemmacompositions}
Suppose that $S:C(K)\rightarrow X_+$ and $T:X_+\rightarrow C(K)$ and 
$(f_m)$
is a pairwise disjoint sequence in $C(K)$ which generates a copy of $c_0$
such that for every $n\in \N$ we have
$$||T\circ I_{-n}\circ P_{-n}\circ S(f_m)||\rightarrow 0,$$
whenever $m\rightarrow \infty$. Then $||T(S(f_m))||\rightarrow 0$.
\end{lemma}

\begin{proof}
Assume that the lemma is false i.e., that $||T(S(f_m))||$s do not converge to $0$.
Then $||S(f_m)||$s do not converge to $0$ and so by \ref{theoremdrewnowski} we may w.l.o.g assume that
$(S(f_m))_{m\in \N}$  generates a copy of $c_0$.

By the assumption that $||T\circ I_{-n}\circ P_{-n}\circ S(f_m)||\rightarrow 0$ and the hypothesis that 
$||T(S(f_m))||$s do not converge to $0$ for every $n\in\N$ we may choose $m_n\in\N$ such that
$||T\circ I_{+n}\circ P_{+n}\circ S(f_{m_n})||$s are separated from $0$ and so also
 $P_{+n}(S(f_{m_n}))$s are separated from zero. 

Applying \ref{factczerodisjoint} we may choose a strictly increasing
sequence $(n_k)$ such that $n_k>k$ and such
that  $ P_{+n_k}\circ S(f_{m_{n_k}})$ generate a copy of $c_0$.

Now we will look for some  basic sequence equivalent to
$(P_{+n_k}((S(f_{m_{n_k}})))$.  
Let $\rho>1$ be a bound for norms of a sequence biorthogonal to $(P_{+{n_k}}(S(f_{m_{n_k}})) )$.

By \ref{lemmauniformity} we can thin-out the sequence $(n_k)\subseteq \N$
so that 
 
$$|d_n (P_n(P_{+n_k}(S(f_{m_{n_k}}))))  |< {1\over 2^{k+2}\rho}$$
holds for all $n,k\in\N$

By the above and the Weierstrass-Stone theorem we can find $(g_{k})$s
which are $(C_n)$-simple, satisfy $P_{-n}(g_n)=0$ for each $n\in\N$  and such that 
$$||g_{k}-I_{+n_k}\circ P_{+n_k}\circ S(f_{m_{n_k}})||<{1\over 2^{k+2}\rho},$$
which means by the Bessaga-Pe\l czy\'nski criterion \ref{theorembessagapelczynski}
that $(g_k)$ generates a copy
of $c_0$.
Thus by \ref{lemmadnsimple} we have that $||T(g_k)||$s converge to $0$. But
$||T(g_k-I_{+n_k}\circ P_{+n_k}\circ S  (f_{m_{n_k}}) ||$s also converge to $0$ which 
means that $T\circ I_{+n_k}\circ P_{+n_k}\circ S(f_{m_{n_k}})$s converge to $0$, a contradiction
with the choice of $m_n$'s. 
\end{proof}

\begin{corollary}\label{corollarycompositions}
Suppose that $S:C(K)\rightarrow X_+$ and $T:X_+\rightarrow C(K)$
are such that for every $n\in\N$ either
\begin{enumerate}
\item   $P_{-n}\circ S$ is weakly compact or
\item  $ T\circ I_{-n}$ is weakly compact. 

\end{enumerate}
 Then $T\circ S: C(K)\rightarrow C(K)$ is weakly compact.

\end{corollary}
\begin{proof} Let $(f_m)$ be a bounded pairwise disjoint sequence in $C(K)$.
We will use \ref{theoremdiesteluhl}, and so we need to prove that
$(T\circ S)(f_m)$ converges to zero.  By \ref{theoremdrewnowski} we may
w.l.o.g. assume that $(f_m)$ generates a copy of $c_0$.
By the hypothesis of the corollary $ T\circ I_{-n}\circ P_{-n}\circ S$ are
weakly compact for each $n\in\N$ and so by  \ref{theoremdiesteluhl}, 
$ T\circ I_{-n}\circ P_{-n}\circ S(f_m)$ converges to zero when $m$ tends to
infinity and $n\in\N$ is fixed. So, we may apply \ref{lemmacompositions}
to conclude that $(T\circ S)(f_m)$ converges to zero and that
$T\circ S$ is weakly compact.

\end{proof}

\begin{remark}\label{remarkcompositions} Suppose that $\Lambda: C(K)\rightarrow X^+$ and
$\Theta: X^+\rightarrow C(K)$ are as in \ref{definitionlambdatheta}.
Note that $P_{-n}\circ \Lambda$ and $\Theta\circ P_{-n}$ are 
finite dimensional and
so weakly compact for all $n\in\N$ but neither $\Lambda$ nor $\Theta$ are weakly compact.
This shows that the composition cannot be replaced by a single operator in the above results.
Note that $\Theta\circ\Lambda=0$.
\end{remark}

\vskip 26pt

\section{The matrix of multiplications of an operator}

\vskip 13pt
\begin{lemma}\label{lemmaweaklycompactmultiplication}
Suppose that $K'$ and $K''$ are homeomorphic  perfect compact spaces and
$i:K''\rightarrow K'$ is a homeomorphism between them.
Let $\phi, \psi \in C(K')$ and $S:C(K')\rightarrow C(K'')$
be a weakly compact operator. Suppose that
for every $f\in C(K')$ we have
$$(\phi f)\circ i=(\psi f)\circ i+S(f).$$
Then $S=0$ and $\phi=\psi$.
\end{lemma}
\begin{proof}
Suppose that $\psi\not=\psi$ and choose a clopen $U\subseteq K'$ or
 and $\varepsilon>0$ such that $|(\phi-\psi)|U|>\varepsilon$.
Let $f_n$s be pairwise disjoint functions of norm one whose supports are included in $U$.
They exist since $K'$ is perfect. Note that $||(\phi-\psi)f_n||\not\rightarrow 0$
and so $||[(\phi-\psi)f_n]\circ i||\not\rightarrow 0$ since $i$ is onto $K'$.
This contradicts the fact that $S(f_n)\rightarrow 0$ by \ref{theoremdiesteluhl}.
Now $S$ must be zero as well.
\end{proof}

\vskip 13pt

Suppose that $T:X_+\rightarrow X_+$ is an operator. By \ref{lemmaexistence} (2)
for each $n,m\in\N$ there are continuous functions $\phi^T_{m,n}\in C(K_n)$
and  weakly compact operators $S^T_{m,n}:C(K_n)\rightarrow C(K_m)$ such that for each $f\in C(K_n)$
$$T(f)|K_m= [\phi^T_{m,n}f]\circ i_{n,m}+S^T_{m,n}(f).$$
We will skip the superscript $T$, if it is clear from the context.
In this section we analyze the matrix $(\phi_{m,n}^T)_{m,n\in\N}$ for an operator
$T$ on $X_+$. Note that Lemma \ref{lemmaweaklycompactmultiplication} implies that
such a decomposition of an
operator  is  unique. 

\vskip 13pt

\begin{lemma}\label{lemmasequenceconverges} Let $T:X_+\rightarrow X_+$
be an operator. Let $n\in \N$ be fixed. The sequence $(\phi^T_{m,n})_{m\in \N}$
converges to $0$.
\end{lemma}

\begin{proof} Fix  $n\in \N$.
If the lemma is false, there are points $x_m\in K_n$ and an $\varepsilon>0$
such that $|\phi_{m,n}(x_m)|>\varepsilon$
for $m$'s from an infinite set $M_1\subseteq \N$.
We may w.l.o.g. assume that $x_m\in K_{1,n}$ or $x_m\in K_{2,n}$
for all $m\in \N$. Say $x_m\in K_{1,n}$, the other case is analogous.
 By the continuity of $\phi_{m,n}$ 
and the fact that $K_n$ is perfect
we can choose pairwise disjoint open sets $U_m\subseteq K_{1,n}$ such that 
$|\phi_{m,n}(x)|>\varepsilon$ holds for
every $x\in U_m$

For $m\in\N$ fixed let $\alpha_{k,m}:K_n\rightarrow \R$
be pairwise disjoint characteristic  functions whose supports $V_{k,m}$  are included
in $U_m$. 
Now we use a characterization \ref{theoremdiesteluhl} of  weakly compact operators on $C(K)$
spaces,  concluding
for every $m\in M_1$ that 
$||S_{m,n}(\alpha_{k,m})||$ converges to zero when
$k$ converges to infinity.
So, for each $m\in M_1$ we can choose $k(m)\in \N$
such that
$$||S_{m,n}(\alpha_{k(m),m})||<\varepsilon/5.$$ 

The definitions of $S_{m,n}$  and the behaviour
of $\phi_{m,n}$ on $U_m$ imply that for $x\in i_{m,n}[V_{k(m),m}]$ we have
$$|T(\alpha_{k(m),m})(x)|>4\varepsilon/5$$
and for all $x\in K_{2,m}$
$$|T(\alpha_{k(m),m})(x)|<\varepsilon/5$$
for each $m\in M_1$.
For $m\in M_1$ pick $y_m\in i_{m,n}[V_{k(m),m}]$ and $z_m\in K_{2,m}$  such that
 $y_m\bowtie z_m$ (see \ref{remarkbowtie}).
Consider $\zeta_m=T^*(\delta_{y_m})|K_n$ and 
$\theta_m=T^*(\delta_{z_m})|K_n$ (restrictions of measures on $K_+$to $K_n$) as functionals on $C(K_n)$.
Now apply \ref{theoremrosenthal1} twice, obtaining an infinite $M_2\subseteq M_1$
such that whenever $M\subseteq M_2$ is an infinite set
such that there is the supremum $\sup_{m\in M}\alpha_{k(m),m} =\alpha$
then $$|\zeta_m(\alpha)-\zeta_m(\alpha_{k(m),m} )|<\varepsilon/5,$$
$$|\theta_m(\alpha)-\theta_m(\alpha_{k(m),m})|<\varepsilon/5.$$ 
for $m\in M$.
Since $|\zeta_m(\alpha_{k(m),m})|>4\varepsilon/5$ and $|\theta_m(\alpha_{k(m),m})|<\varepsilon/5$ in
terms of $T$ we obtain that
for $m\in M$ we have 
$$|T(\alpha)(z_m)|>3\varepsilon/5$$
$$|T(\alpha)(y_m)|<2\varepsilon/5$$
Of course the above supremus $\alpha$ exist in $C(K_n)$ by \ref{lemmaexistence} (5)
As $T(\alpha)$ is in $X_+$  and $y_m\bowtie z_m$ this contradicts \ref{corollary2assymptotic}
and completes the proof of the lemma.
\end{proof}

\begin{lemma}\label{lemmablocks} Suppose
that $T:X_+\rightarrow X_+$ is a linear bounded operator.
Suppose that $m\in \N$, $F\subseteq \N$ is a finite set
such that  $||\sum _{n\in F}|\phi^T_{m,n}\circ i_{n,m}|||>\varepsilon$.
Then there is $f\in C(\bigcup_{n\in F}K_n)$ which is $(C_n)$-simple and of norm one such that 
$||T(f)|K_m||>\varepsilon.$
\end{lemma}

\begin{proof} Let $x\in K_m$ be such that $\sum _{n\in F}|\phi^T_{m,n}\circ i_{n,m}|(x)>\varepsilon$.
Either  $x\in K_{1,m}$ or  $x\in K_{2,m}$. Say $x\in K_{1,m}$,
the other case is analogous.
By the continuity of $\phi_{m,n}$ 
and the fact that $K_n$ is perfect
we can find clopen $U\subseteq K_{1,m}$ such that 
$\sum_{n\in F}|\phi_{m,n}\circ i_{n,m}(x)|>\varepsilon+\varepsilon'$ holds for
every $x\in U$ and some $\varepsilon'>0$ and such that no $\phi_{m,n}$ for $n\in F$ changes its sign on $i_{n,m}[U]$
(if $\phi_{m,n}(x)=0$ we may remove $n$ from $F$).
Let $V_{l}\subseteq U\subseteq K_{1,m}$ for $l\in\N$ be pairwise disjoint. By
the density
of   $B_m$ in $Clop(K_{1,m})$ we may choose $V_l$s to be in $B_m$.
Let $$\alpha_{n,l}=\delta_n\chi_{i_{n,m}[V(l)]}$$
 where $\delta_n=\pm 1$ and its sign  is the same as $\phi_{m,n}$ on $i_{n,m}[U]$
so 
$$\sum_{n\in F}(\phi_{m,n}\alpha_{n,l})\circ(i_{n,m}(x))>\varepsilon+\varepsilon'.$$ 
for each $x\in V_{l}$ and each $l\in \N$.

Let $W_{l}=j[V_l]\subseteq K_{2,m}$ and let $\beta_{n,l}=\delta_n\chi_{i_{n,m}[W_l]}$
Then, for each $l\in \N$
 $$f_l=\sum_{n\in F}\alpha_{n,l}+\beta_{n,l}$$
is $(C_n)$-simple,  $f_l\in C(\bigcup_{n\in F}K_n)$ and 
$$\sum_{n\in F}(\phi_{m,n}f_l)(i_{n,m}(x))>\varepsilon+\varepsilon'.$$ 
for each $x\in V_{l}$. Moreover  $f_l$s are pairwise disjoint.

Now we use a characterization \ref{theoremdiesteluhl} of a weakly compact operators on $C(K)$
spaces,  
so 
 we have that 
$||\sum_{n\in F}S_{m,n}(\alpha_{n,l}+\beta_{n,l})||$ converges to zero when
$l$ converges to infinity.
So, there is an $l_0\in \N$ such that 
such that $$||\sum_{n\in F}S_{m,n}(\alpha_{n,l_0}+\beta_{n,l_0})|| <\varepsilon'$$ 
So, for $x\in V_l$ we have 
$$|T(f_{l_0})(x)|\geq
|\sum_{n\in F}(\phi_{m,n}f_{l_0})(i_{n,m}(x))|- ||\sum_{n\in F}S_{m,n}(\alpha_{n,l_0}+\beta_{n,l_0})||
\geq \varepsilon$$
and hence $||T(f)|K_m||>\varepsilon$ for $f=f_{l_0}$.
\end{proof}

\begin{lemma}\label{lemmaseriesconverges} Let $T:X_+\rightarrow X_+$
be an operator. Let $m\in \N$ be fixed. The series $\sum _{n\in \N}|\phi^T_{m,n}\circ i_{n,m}|$
converges
 uniformly.
\end{lemma}

\begin{proof} Fix  $m\in \N$. If the series does not converge uniformly, then
there is $\varepsilon>0$ and  there are
strictly increasing $n_k\in\N$ and $x_k\in K_m$ for $k\in \N$ such that
$$\sum_{n_k\leq n<n_{k+1}}|\phi_{m,n}(i_{n,m}(x_k))|>\varepsilon.$$
We may w.l.o.g. assume that $k<n_k$. 

By \ref{lemmablocks} we can find $f_k\in C(\bigcup_{n_k\leq n<n_{k+1}}K_n)$ 
which are $(C_n)$-simple and of norm one such that $||T(f_k)||>\varepsilon$.
As supports of $f_k$s are disjoint they generate a copy of $c_0$ by \ref{factczero} .
So we obtain a contradiction with \ref{lemmadnsimple}. 
\end{proof}

\begin{lemma}\label{lemmaseriesmultiplication}
Let $T:X_+\rightarrow X_+$
be an operator and $f=(f_n)\in X_+$. The series $\sum _{n\in \N}[\phi^T_{m,n}f_n]\circ i_{n,m}$
converges uniformly to a function $g\in C(K_m)$ satisfying $||g||\leq ||T||||f||$.
\end{lemma}
\begin{proof}
Note that for each $x\in K_m$
$$|\sum_{n_k\leq n<n_{k+1}}[\phi_{m,n}f_n](i_{n,m}(x))|\leq 
 ||f|| (\sum_{n_k\leq n<n_{k+1}}|\phi_{m,n}(i_{n,m}(x))|)\leq$$
$$\leq||f|| ||\sum_{n_k\leq n<n_{k+1}}|\phi_{m,n}\circ i_{n,m}|||,$$
and so by Lemma \ref{lemmaseriesconverges} the partial sums of $\sum _{n\in \N}[\phi^T_{m,n}f_n]\circ i_{n,m}$
satisfy the   uniform Cauchy condition and so $\sum _{m\in \N}[\phi^T_{n,m}f_n]\circ i_{n,m}$
is uniformly convergent. Call the limit $g\in C(K_m)$.

For the proof of the second part of the lemma,  note that if $||g||> ||T||||f||$, for some
$f\in X_+$, then by the above inequalities
$\sum_{1\leq n\leq n_1}|\phi_{n,m}(i_{n,m}(x))|>||T||$ but this would give
an $f'\in C(\bigcup_{1\leq n\leq n_1}K_n)$ of norm one such that $||T(f')||>||T||$ by \ref{lemmablocks},
 a contradiction.

\end{proof}

\begin{lemma}\label{lemmamultiplication} Let $T:X_+\rightarrow X_+$ be a bounded linear operator.
There is a bounded linear  operator $\Pi^T: X_0\rightarrow X_0$  which satisfies
$$\Pi^T(f)|K_m=[\phi^T_{m,n}f]\circ i_{n,m}$$
for $f\in C(K_n)$ and whose norm is not bigger than $||T||$.
\end{lemma}
\begin{proof} Let $f\in X_0$ and $f_n=f|K_n$. Fix $n\in\N$, 
by \ref{lemmaseriesmultiplication}
for every $m\in\N$ the series  $\sum_{n\in \N}[\phi_{m,n}f_n]\circ i_{n,m}$
converges uniformly and its norm is not bigger than $||T||||f||$. So
for $f=(f_n)\in X_0$ we can define
$\Pi^T(f)$ by requiring
$$\Pi^T(f)|K_m=[\sum_{n\in \N}\phi^T_{m,n}f_n]\circ i_{n,m}$$
which agrees with the condition of the lemma for $f\in C(K_n)$. It is a bounded operator
into $X_\infty$ whose norm is not bigger than $||T||$ by \ref{lemmaseriesconverges}. 
However, its  restriction to any $C(K_n)$ is into
$X_0$ by Lemma \ref{lemmasequenceconverges} so on $X_0$, it is into $X_0$.

\end{proof}

\begin{lemma}\label{lemmaweaklycompactreminders} Let $T,R: X_+\rightarrow X_+$ be any operators and $m,n\in\N$.
The operators from $C(K_n)$ into $C(K_m)$ given by 
\begin{itemize}
\item $P_{m} (\Pi^T-T) RI_{n}$,
\item $P_{m} (\Pi^R-R) T I_{n}$,

\item $P_{m} (\Pi^T-T) (\Pi^R-R) I_{n}$,   
\item $P_{m} (\Pi^R-R)(\Pi^T-T) I_{n}$,    

\item $P_{m}R (\Pi^T-T) I_{n}$

\item $P_{m}T (\Pi^R-R) I_{n}$,  
\end{itemize}
  are all weakly compact.
\end{lemma}
\begin{proof} 
Note that for each $k\in\N$ the operators
$$[P_{m}(\Pi^T-T)]I_{-k}= \sum_{l\leq k} S^T_{m, l}$$
$$[P_{m}(\Pi^R-R)]I_{-k}= \sum_{l\leq k} S^R_{m, l}$$
$$P_{-k} [(\Pi^T-T) I_{n}] = \sum_{l\leq k} S^T_{l, n}$$
$$P_{-k} [(\Pi^R-R) I_{n}] = \sum_{l\leq k} S^R_{l, n}$$
are all weakly compact because the right hand sides are finite sums of weakly compact
operators.
So by \ref{corollarycompositions} we may conclude the lemma.
\end{proof}

\begin{remark}\label{remarkmultiplications} If $\Lambda: C(K)\rightarrow X^+$ and $\Theta:X^+\rightarrow C(K)$
are as in \ref{definitionlambdatheta}
then they  decompose  so that
all $\phi_{n,*}^\Lambda$ and $\phi_{*,m}^\Theta$ are zero. This proves that
the approximation of a single operator $T$ by $\Pi^T$ is not always modulo a weakly
compact operator.
\end{remark}

\begin{corollary} Let 
Suppose that $T, R:X_+\rightarrow X_+$ are such operators that
$R\circ T=0_{C(K_{1,0})}\oplus Id_{Y_+}$ and $T\circ R=Id_{X_+}$. Then 
we have $\Pi^R\Pi^T=0_{C(K_{1,0})}\oplus Id_{Y_0}$and $\Pi^T\Pi^R=Id_{X_0}$
\end{corollary}
\begin{proof}

Note that we have 
$$\Pi^T\Pi^R=(\Pi^T-T+T)(\Pi^R-R+R)=$$
$$(\Pi^T-T) (\Pi^R-R)+(\Pi^T-T)R+T(\Pi^R-R)+TR$$
and 
$$\Pi^R\Pi^T =(\Pi^R-R+R)(\Pi^T-T+T) =$$
 $$=(\Pi^R-R)(\Pi^T-T) + (\Pi^R-R)T  +R(\Pi^T-T) +RT$$
So by Lemma \ref{lemmaweaklycompactreminders}
we have $$P_m\Pi^T\Pi^R I_n=S_{m,n}'+ P_mI_n$$
for all $m,n\in\N$ where $S_{m,n}'$s are all weakly compact and 
$$P_m\Pi^R\Pi^T I_n={S''}_{m,n}+ P_mI_n$$
for all $n,m\in \N\setminus\{0\}$ where ${S''}_{m,n}$ are all weakly compact.
Moreover 
 $$P_0\Pi^R\Pi^T I_0={S''}_{0,0}+ P_0 0_{C(K_{1,0})}\oplus 
Id_{Y_+}I_0={S''}_{0,0}+0_{C(K_{1,0})}\oplus Id_{C(K_{2,0})}$$

Of course $P_mI_n$ is zero or the identity operator depending whether $m=n$
but in both cases it is a multiplication by a continuous function, 
also multiplying by $0$ on $K_{1,0}$ and by $1$ on $K_{2,0}$ is
a multiplication by a continuous function, so
by \ref{lemmaweaklycompactmultiplication} 
all the operators $S_{m,n}'$  and ${S''}_{m,n}$ are  zero operators concluding the lemma.

\end{proof}

\begin{corollary}\label{corollaryzeroone} Suppose that $T, R:X_+\rightarrow X_+$ are such operators that
$R\circ T=0_{C(K_{1,0})}\oplus Id_{Y_+}$ and $T\circ R=Id_{X_+}$. Then 
we have the following:
\begin{itemize}
\item $\sum_{k\in\N} (\phi^R_{m,k}\circ i_{k,n})\phi_{k,n}^T(x)=0$ for distinct $n,m\in \N$ and
all $x\in K_n$,
 \item $\sum_{k\in\N} (\phi^R_{n,k}\circ i_{k,n})\phi_{k,n}^T(x)=1$ for all 
$(n,x)\in (\N\setminus\{0\}\times K_+\setminus K_{0}) \cup (\{0\}\times K_{2,0})$,
 \item $\sum_{k\in\N} (\phi^R_{0,k}\circ i_{k,0})\phi_{k,0}^T(x)=0$ for all $x\in K_{1,0}$,

\item $\sum_{k\in\N} (\phi^T_{m,k}\circ i_{k,n})\phi_{k,n}^R(x)=0$ for distinct $n,m\in \N$ and
all $x\in K_n$,

\item $\sum_{k\in\N} (\phi^T_{n,k}\circ i_{k,n})\phi_{k,n}^R(x)=1$ for all $n\in\N$ and all $x\in K_n$.

\end{itemize}

\end{corollary}

\begin{lemma}\label{lemmamultiplicationpointwise}
Let $x\in K$ and $x_n=i_{n,*}(x)$
and $A=\{x_n:n\in \N\}$.
Let $T:X_+\rightarrow X_+$ be a bounded linear operator.
There is a bounded linear  operator $\Pi^T_{x}: c_0(A)\rightarrow c_0(A)$  which satisfies
$$\Pi^T_{x}(1_{x_n})(x_m)= \phi^T_{m,n}(x_n),$$
and whose norm is not bigger than $||T||$.
\end{lemma}
\begin{proof} Note that if ${\vec{1}}_n$ denotes the constant function
on $K_n$ whose value is one, then $\Pi^T({\vec 1}_n)(x_m)=\phi^T_{m,n}(x_n)$.
Since the range of $\Pi^T$ is $X_0$ we may conclude that $(\phi^T_{m,n}(x_n))_{m\in \N}$
belongs to $c_0(A)$ and its norm is not bigger than $||T||$ by \ref{lemmamultiplication}. So the above formula
 well defines $\Pi^T_{x}: c_0(A)\rightarrow c_0(A)$.
\end{proof}

\begin{lemma}\label{lemmaisomorphismpointwise}
Suppose that $T, R:X_+\rightarrow X_+$ are such operators that
$R\circ T=0_{C(K_{1,0})}\oplus Id_{Y_+}$
 and $T\circ R=Id_{X_+}$. Let $x\in K_{1,*}$, $y\in K_{2,*}$, $x_n=i_{n,*}(x)$
$y_n=i_{n,*}(y)$, $A=\{x_n:n\in \N\}$, $A'=\{x_n:n\in \N\setminus\{0\}\}$, $B=\{y_n:n\in \N\}$.  
Let 
$$\Pi_{x,y}^T= (\Pi_{x}^T, \Pi_{y}^T): c_0(A\cup B)\rightarrow c_0(A\cup B),$$
$$\Pi_{x,y}^R= (\Pi_{x}^R, \Pi_{y}^R): c_0(A\cup B)\rightarrow c_0(A\cup B).$$
Then $\Pi^R_{x,y} \circ\Pi^T_{x,y} =
0_{x_0}\oplus Id_{c_0(A'\cup B)}$ and $\Pi^T_{x,y} \circ \Pi^R_{x,y} =Id_{c_0(A\cup B)}$. 

\begin{proof}
Apply corollary \ref{corollaryzeroone} pointwise.
\end{proof}

\end{lemma}

\vskip 26pt

\section{Detecting shifts in isomorphisms.}
\vskip 13pt
\noindent If $I$ is a set and $f\in c_0(I)$, by support of
$f$ we mean $\{i: f(i)\not=0\}$. 
\begin{lemma}\label{lemmagowers}
Let $A=\{x_n:n\in N\}$ and $B=\{y_n:n\in N\}$.
Suppose that $T:c_0(A\cup B)\rightarrow c_0(A\cup B)$ is an operator
such that $T(1_{x_0})=0$,
such that $T|(\{f\in c_0(A): f(0)=0\})$ is an isomorphism onto $c_0(A)$ 
and $T_2=T|(c_0(B)$ is an isomorphism onto $c_0(B)$. Then there are
$f_k \in c_0(A\cup B)$ of pairwise disjoint and finite supports and there is an $\varepsilon>0$ for which 
$$f_k(x_n)=f_k(y_n)$$ 
for every $n\in N$ and for some distinct $m_k$s we have
$$|T(f_k)(x_{m_k})-T(f_k)(y_{m_k})|>\varepsilon.$$
\end{lemma}
\begin{proof}
Let $T_1=T|c_0(A)$.
Consider $T_3: c_0(B)\rightarrow c_0(B)$
so that we have $T_3(1_{\{y_n\}})(y_m)=T_1(1_{\{x_{n}\}})(x_m)$ for every $m,n\in N$. That is we
consider a copy of $T_1$ copied 
from $c_0(A)$ on $c_0(B)$.\par
\noindent {\bf Claim:} $T_3-T_2: c_0(B)\rightarrow c_0(B)$ is not weakly compact.\par
\noindent Proof of the Claim: If it is, then it is strictly singular
by \ref{theorempelczynski}, and so we will
be able to apply the Fredholm theory (see \cite{lindenstrausstzafriri}). Namely, then,
the Fredholm index of $T=(T_1, T_2): c_0( A\cup B)\rightarrow c_0(A\cup B)$ 
is equal to the Fredholm index of $(T_1, T_2)+(0,T_3-T_2)=(T_1,T_3)$ which must be an
even integer  since $T_1$ and
$T_3$ are copies of the same operator. However $T_2$ is an isomorphism and $T_1$ is onto
and  has kernel
of dimension one, and so $T=(T_1, T_2)$ must have odd Fredholm index,
a contradiction which completes the proof of the claim.\par
\vskip 6pt
\noindent Now use \ref{theoremdiesteluhl} to find 
a pairwise disjoint $e_k\in c_0(B)$ such that  there are $m_k$s such that
$$|[(T_3-T_2)(e_k)](y_{m_k})|>\varepsilon$$
for some $\varepsilon>0$. 
Since the operator is bounded and the norms of sums of $e_k$s are bounded as well
since there are disjoint, we may assume that all $m_k$s are distinct.

\noindent Now define $f_k\in c_0(A\cup B)$
by  $f_k|B=e_k$ and  $f_k(x_n)=e_k(y_n)$ for each
$n\in N$. So we have $f_k(x_n)=f_k(y_n)$ for each $k,n\in N$.
Also  
$$T(f_k)(x_{m_k})= T_1(f_k| A)(x_{m_k})=
T_3(e_k))(y_{m_k}),$$
$$T(f_k)(y_{m_k})= T_2(e_k))(y_{m_k})$$
so we obtain
$$|T(f_k)(x_{m_k})-T(f_k)(y_{m_k})|>\varepsilon $$
 as required.\par
\end{proof}

\begin{lemma}\label{lemmaweaklycompacttrick}
Suppose that $S:C(K_n)\rightarrow C(K_m)$ is a weakly compact operator.
Let $\tau_n:K_n\rightarrow L_n$ be the canonical surjection as in Section 3. Then
for each $j\in\N$
$$\Gamma(S,j)=\{t\in L_n: \exists x\in \tau^{-1}_n(\{t\}) \ \ 
|S^*(\delta_{i_{m,n}(x)})(\tau_n^{-1}(\{t\}))|>1/j\}$$
is finite.
\end{lemma}

\begin{proof}
If $\Gamma(S,j)\subseteq L_n$ is infinite, we can choose
a discrete sequence $\{t_l:l\in\N\}\subseteq \Gamma(S,j)$. Let 
$x_l\in K_n$ be such that $\tau_n(x_l)=t_l$  and 
$|S^*(\delta_{i_{m,n}(x_l)})(\tau_n^{-1}(\{t_l\}))|>1/j$.
Now use the fact that
$$\tau^{-1}[\{t\}]=\bigcap\{\tau^{-1}([a]): t\in [a], \ a\in C_n\}$$
and the family whose intersection appears above is directed to conclude
that for each $l\in \N$ there is an $a_l\in C_n$ such that $t_l\in [a_l]$ and 
$|S^*(\delta_{i_{m,n}(x_l)})(\tau_n^{-1}[[a_l]])|>1/j$.
Moreover, as $\{t_l:l\in\N\}$ is discrete we may w.l.o.g. assume that
$a_l$'s are pairwise disjoint. This means that
$$|S(\chi_{\tau^{-1}_n[[a_l]]})(i_{m,n}(x_l))|>1/j$$
for each $l\in \N$, which contradicts \ref{theoremdiesteluhl} since
$a_l$'s are pairwise disjoint. 

\end{proof}

\begin{theorem} $Y_+$ and $X_+$ are not isomorphic.
\end{theorem}
\begin{proof}
Suppose $T_1: Y_+\rightarrow X_+$ and $R_1: X_+\rightarrow Y_+$ 
be  mutually inverse isomorphisms.
Define $T,R:X_+\rightarrow X_+$ by $R=R_1$ and
$T =0_{C(K_{1,0})}\oplus T_1$.
We have $R\circ T= 0_{C(K_{1,0})}\oplus Id_{Y_+}$
and $T\circ R=Id_{X_+}$.

By \ref{lemmaisomorphismpointwise} 
$\Pi_{x,y}^T$ satisfies the hypothesis of \ref{lemmagowers}
To use it successfully we still need to
make  appropriate choice of $x$ and $y$. Let
$$\Gamma=\{t\in L: \exists n,m, j\in \N\ h_{n,*}(t)\in \Gamma(S_{m,n},j)\}.$$
By \ref{lemmaweaklycompacttrick}
 $\Gamma$ is countable where $S_{m,n}=S_{m,n}^T$.
Pick $t\in L\setminus \Gamma$ and $x\in K_1$, $y\in K_2$ such that
$\tau(x)=\tau(y)=t$ (see \ref{remarkbowtie}). As usual let $x_n=i_{n,*}(x)$ and $y_n=i_{n,*}(y)$,
and $t_n=h_{n,*}(t)$

Now use \ref{lemmagowers} obtaining 
$\alpha_k \in c_0(A\cup B)$ of pairwise disjoint and finite supports $F_k$ and an $\varepsilon>0$ such that
$\alpha_k(x_n)=\alpha_k(y_n)$ 
for every $n\in N$ and for some distinct $m_k$s 
$$|\Pi^T_{x,y}(\alpha_k)(x_{m_k})-\Pi^T_{x,y}(\alpha_k)(y_{m_k})|>\varepsilon.$$

For $n\in F_k$ we will find $a_n\in C_n$ such that 
the operator $T$ will have
approximately the same behavior on vectors 
$$f_k=\sum_{n\in F_k}\alpha_k(x_n)\chi_{\tau_n^{-1}[[a_n]]}\in X_0$$
as $\Pi_{x,y}^T$ on $\alpha_k$s.

Note that for each $z_n$ satisfying $\tau(z_n)=t_n$
(in particular $z_n=x_n,y_n$),
by the choice of $t$ from outside $\Gamma$, the measure $S_{m,n}^*(\delta_{i_{m,n}(z_n)})$ 
of the set $\tau^{-1}_n[\{t_n\}]$ is zero. 
So, using the fact that
$\tau^{-1}_n[\{t_n\}]=\bigcap\{\tau^{-1}([a]): t_n\in [a]\ \ a\in C_n\}$
for each $n\in F_k$ find $a_n\in C_n$
such that $t_n\in [a_n]$ and
$$|S^*_{m_k,n}(\delta_{x_{m_k}})(\tau_n^{-1}([a_n]))|,
|S^*_{m_k,n}(\delta_{y_{m_k}})(\tau_n^{-1}([a_n]))|<{\varepsilon\over 3|F_k||\alpha_k(x_n)|}$$
so, for each $k\in\N$ and $n\in F_k$ we have 
$$|\alpha_k(x_n) S_{m_k,n}(\chi_{\tau_n^{-1}[[a_n]]})(x_{m_k})|,
|\alpha_k(x_n) S_{m_k,n}(\chi_{\tau_n^{-1}[[a_n]]})(y_{m_k})|<{\varepsilon\over 3|F_k|}$$
and so
$$|S_{m_k,n}(\sum_{n\in F_k}\alpha_k(x_n)\chi_{\tau_n^{-1}[[a_n]]}(x_{m_k})|,
|S_{m_k,n}(\sum_{n\in F_k}\alpha_k(x_n)\chi_{\tau_n^{-1}[[a_n]]}(y_{m_k})|<{\varepsilon\over 3}$$
and finally
$$|S_{m_k,n}(f_k)(x_{m_k})|,\ |S_{m_k,n}(f_k)(y_{m_k})|<\varepsilon/3$$
for each $k\in\N$ and each $n\in F_k$.
On the other hand 
$$|\Pi^T(f_k)(x_{m_k})-\Pi^T(f_k)(y_{m_k})|=
|\Pi^T_{x,y}(f_k)(x_{m_k})-\Pi^T_{x,y}(f_k)(y_{m_k})|>\varepsilon,$$
hence for each $k$ we have
$$|T(f_k)(x_{m_k})-T(f_k)(y_{m_k})|>\varepsilon/3$$
where $m_k$s are distinct and $f_k$'s are bounded  $(C_n)$-simple and of disjoint supports. 

Now note that for any infinite $M\subseteq \N$ there
is in $X_+$ the supremum $f_M$ of $\{f_k:k\in\N\}$.

Let $\zeta_k=T^*(\delta_{x_{m_k}}-\delta_{y_{m_k}})$. In particular we have
$|\zeta_k(f_k)|>\varepsilon/3$. Apply \ref{corollarysupremus} to obtain
an infinite $M$ such that
$|\zeta_k(f_M)|>\varepsilon/4$ for each $k\in M$.
But this means that 
$$|T(f_M)(x_{m_k})-T(f_M)(y_{m_k})|>\varepsilon/4$$
for each $k\in M$.
This is the final contradictions with \ref{corollary2assymptotic} since
$x_{m_k}\bowtie y_{m_k}$.

\end{proof}

\begin{corollary} $X_+$ and $Y_+$ are not isomorphic but
each is isomorphic to a complemented subspace of the other.
\end{corollary}

\vskip 50pt

\bibliographystyle{amsplain}

\end{document}